\documentclass[11pt]{amsart}

\newcommand{\Addresses}{{
  \bigskip
  \footnotesize  
  
  \noindent Gabriele Viaggi, \textsc{Mathematical Institute, University of Pisa, Pisa}\par\nopagebreak
  \textit{E-mail address}: \texttt{gabriele.viaggi@unipi.it}
  }}

\usepackage[margin=1.5in]{geometry}
\usepackage[colorlinks,citecolor=cyan,linkcolor=blue]{hyperref}

\usepackage{dsfont}
\usepackage{amsmath}
\usepackage{amsthm}
\usepackage{amssymb}
\usepackage[all,cmtip]{xy}
\usepackage{enumerate}
\usepackage{enumitem}
\usepackage{overpic}
\usepackage{float}

\newcommand{\ep}{
\epsilon
}
\newcommand{\mc}[1]{
\mathcal{#1}
}
\newcommand{\mb}[1]{
\mathbb{#1}
}

\newtheorem*{thm*}{Theorem}
\newtheorem*{cor*}{Corollary}
\newtheorem*{pro*}{Proposition}

\newtheorem{thmA}{Theorem}

\newtheorem{corA}[thmA]{Corollary}

\newtheorem{thm}{Theorem}[section]

\newtheorem{lem}[thm]{Lemma}
\newtheorem{pro}[thm]{Proposition}

\theoremstyle{definition}

\newtheorem*{defi*}{Definition}
\newtheorem{dfn}[thm]{Definition}

\title{Linear volume bounds for drilling and filling}
\author{Gabriele Viaggi}

\begin{document}

\begin{abstract}
We prove uniform linear bounds on the volume variation under drilling and filling operations on finite volume hyperbolic 3-manifolds. 
\end{abstract}

\maketitle

\section{Introduction}

There are two important operations that one can perform on a finite volume hyperbolic 3-manifold. One can drill out a simple closed geodesic and one can Dehn fill a cusp. It is a well-known phenomenon that all drillings and most fillings give back a hyperbolizable 3-manifold whose volume is larger for the first operation and smaller for the second. The goal of this short note is to prove two results that give a {\em linear} quantitative control on how much the volume can change.

\subsection*{Drilling}
Let us start with drillings. We prove the following.

\begin{thmA}
\label{thm:main}
There exists a constant $c>0$ such that the following holds. Let $M$ be a closed hyperbolic 3-manifold and let $\gamma\subset M$ be a simple closed geodesic. Suppose that $\gamma$ has length $\ell$ and has an embedded tubular neighborhood of radius at least $R\in(0,1)$. Then the unique hyperbolic metric on $M-\gamma$ has volume
\[
{\rm vol}(M-\gamma)\le{\rm vol}(M)+c\frac{\ell}{R}.
\]
\end{thmA}

\begin{corA}
\label{cor:main}
There exists a constant $c>0$ such that the following holds. Let $M$ be a closed hyperbolic 3-manifold and let $\gamma\subset M$ be a shortest non-trivial closed geodesic in $M$. Then the unique hyperbolic metric on $M-\gamma$ has volume
\[
{\rm vol}(M-\gamma)\le{\rm vol}(M)+c\ell
\]
where $\ell=2\cdot{\rm inj}(M)$ is the length of $\gamma$.
\end{corA}

Similar estimates on the volume increase under drilling operations on hyperbolic 3-manifolds were already considered by Bridgeman \cite{B}, Agol \cite{A}, Agol, Storm, Thurston, and Dunfield \cite[Theorem 10.1]{AST} where they proved the existence of explicit functions $f,g>1$ such that ${\rm vol}(M-\gamma)\le f(R){\rm vol}(M)+g(R)\ell$. So, the main achievement of this note is to get rid of the multiplicative constant in front of ${\rm vol}(M)$. This shows that the increase in volume only depends on the {\em local} geometry around $\gamma$. 

Theorem \ref{thm:main} goes in the direction of a conjecture of Bridgeman \cite{B}, which asks whether there is a linear upper bound on the volume increase under drilling operations. While very good estimates are available when the drilled geodesics are short (see the work of Hodgson and Kerckhoff \cite{HK}), it is still unclear what happens for long geodesics. Note that as long as we can uniformly bound from below $R$ the increase in volume is at most linear. This leads immediately to Corollary \ref{cor:main}. In fact, it is well-known that a shortest geodesic in a closed hyperbolic 3-manifold has a tubular neighborhood of radius at least $\ell/4$ where $\ell$ is the length of the geodesic (see \cite[Proposition 1.11]{GMT}). Combined with standard consequences of Margulis Lemma, this provides a universal constant $R_0>0$ such that a shortest geodesic in a closed hyperbolic 3-manifold $M$ always has a tubular neighborhood of radius $R\ge R_0$ (for much more refined estimates see \cite[Theorem 4.1]{GMT}). 

Let us mention the fact that the restriction to simple geodesics in closed hyperbolic 3-manifolds, rather than dealing with geodesic links in arbitrary finite volume ones, is not important. A straightforward generalization of Theorem \ref{thm:main} holds in this larger setting with exactly the same arguments and more robust bookkeeping. We chose to restrict ourselves to the simple case to streamline the exposition.

\subsection*{Filling}
Now we move to fillings. We give a new proof of the following lower bound due to Hodgson and Kerckhoff \cite{HK} using completely different techniques.

\begin{thmA}
\label{thm:main'}
There exists a constant $c>0$ such that the following holds. Let $M$ be a finite volume hyperbolic 3-manifold with one cusp. Let $C$ be a cusp neighborhood of $M$ bounded by the horospherical torus $\partial C$. Consider a simple closed geodesic $\mu\subset\partial C$ for the intrinsic flat metric on $\partial C$. If the length of $\mu$ is $\ell>\sqrt{c}\pi$ and the area of $\partial C$ is $A$ then the Dehn filling of $M-C$ with slope $\mu$, denoted by $M_\mu$, has a hyperbolic metric of volume
\[
{\rm vol}(M_\mu)+\frac{A}{2}\frac{\pi^2}{\ell^2}\left(1-c\frac{\pi^2}{\ell^2}\right)\le{\rm vol}(M).
\]
\end{thmA}

We remark that our assumptions differ slightly from the ones of \cite{HK} as we ask that $\ell\ge\sqrt{c}\pi$ whereas \cite{HK} requires that $\ell^2/A$ is bigger than a universal constant. 

The problem of understanding how much the volume decreases under filling is better understood compared to the drilling case and has been studied with various methods. Thurston \cite[Chapter 6]{ThuNotes} showed that ${\rm vol}(M_\mu)<{\rm vol}(M)$ and that for every {\em fixed} $M$ we have ${\rm vol}(M_\mu)\to{\rm vol}(M)$ as $\ell\to\infty$. Neumann and Zagier \cite{NZ} then clarified, for every {\em fixed} $M$, the asymptotic behavior of ${\rm vol}(M)-{\rm vol}(M_\mu)$ as a function of $A/\ell^2$. Later, Hodgson and Kerckhoff \cite{HK}, gave a {\em universal} bound of the form ${\rm vol}(M)-{\rm vol}(M_\mu)= A\pi^2/2\ell^2+O(A^2/\ell^4)$.

The restriction to Dehn fillings of hyperbolic 3-manifolds of finite volume with one cusp rather than allowing multiple Dehn fillings of a selection of cusps of an arbitrary finite volume hyperbolic 3-manifold is not really important. A straightforward generalization of Theorem \ref{thm:main'} holds in that larger setting with the same proof but more notation. In order to keep the exposition as linear and short as possible, we restricted ourselves to the simple setup.

\subsection*{On the proofs}
Lastly, a few words on the proofs. 

To prove Theorem \ref{thm:main}, we follow the approach of \cite[Theorem 10.1]{AST} with a difference. We first change the metric of $M$ inside a tubular neighborhood $T$ of $\gamma$ interpolating between the metric on the tube and the metric on a suitable hyperbolic cusp. We do so by keeping under control the scalar curvature (see Lemma \ref{lem:functiondrill}). Then we compare the volume of the modified metric with that of the hyperbolic structure on $M-\gamma$. Instead of relying on Perelman's monotonicity formula, we exploit some other consequences of Perelman's work \cite{Pe1,Pe2}, namely, the exact computation of the sigma invariant of closed hyperbolic 3-manifolds as achieved by Anderson \cite{And} and Kleiner and Lott \cite{KL}. This allows us to keep track of the fact that we {\em did not change} the metric outside the tubular neighborhood $T$ and obtain an additive error rather than a multiplicative one in the volume increase. 

The proof of Theorem \ref{thm:main'} follows exactly the same strategy. We change the metric of $M$ in a cusp neighborhood $C$ as in Gromov-Thurston's $2\pi$-Theorem \cite[Theorem 9]{BH} (of which we prove a quantitative version, see Proposition \ref{pro:metric on tube}). This time the interpolation between the metric on the cusp and the metric on a tube is more delicate as we want to keep under control the scalar curvature up to higher order (see Lemma \ref{lem:functionfill}). Then we use the sigma invariant comparison to conclude.

\subsection*{Acknowledgements}
I warmly thank Ian Agol and Martin Bridgeman for useful discussions and generous feedback on a first draft of this article.

\section{Preliminaries}
This section reviews some crucial notions and tools that will be used in the proof of the main results. It is divided in two parts. The first part introduces the so-called sigma invariant of a closed 3-manifold and gives a useful integral estimate (Proposition \ref{pro:integral scalar}). In the second part we summarize the curvature properties of a special family of metrics that can be used to describe both hyperbolic tubes and hyperbolic cusps.

\subsection{The sigma invariant}

The sigma invariant of a 3-manifold is the key quantity involved in the proofs of Theorems \ref{thm:main} and \ref{thm:main'}.

\begin{dfn}[Sigma Invariant]
Let $M$ be a closed 3-manifold. For every Riemannian metric $(M,g)$ define the Einstein-Hilbert functional     
\[
\mc{E}(g):=\frac{\int_M{S_g\,{\rm dvol}_g}}{{\rm vol}(M,g)^{1/3}}
\]
where $S_g$ is the scalar curvature of $(M,g)$.

Denote by $[g]$ the conformal class of $(M,g)$ that is the set of Riemannian metrics on $M$ conformally equivalent to $g$. The classical Yamabe constant of $(M,[g])$ is given by
\[
Y(M,[g]):=\inf_{g'\in [g]}{\mc{E}(g')}.
\]

By considering the set $\mc{C}$ of all possible conformal classes of Riemannian metrics on $M$, one defines the sigma invariant as
\[
\sigma(M):=\sup_{[g]\in\mc{C}}{Y(M,[g])}.
\]
\end{dfn}

The following gives a useful lower bound.

\begin{pro}
\label{pro:integral scalar}
Let $M$ be a closed 3-manifold. For every Riemannian metric $(M,g)$ we have
\[
Y(M,[g])\ge -\left(\int_M{|S_g|^{3/2}{\rm dvol}_g}\right)^{2/3}.
\]    
\end{pro}

\begin{proof}
Consider an arbitrary conformal change of the form $g'=u^4g$ where $u:M\to(0,\infty)$ is a smooth function. We have
\[
S_{g'}=u^{-5}(-8\Delta_gu+S_gu)\quad\text{\rm and }\quad{\rm dvol}_{g'}=u^6{\rm dvol}_g.
\]
Therefore
\begin{align*}
\frac{\int_M{S_{g'}{\rm dvol}_{g'}}}{{\rm vol}(M,g')^{1/3}} &=\frac{\int_M{-8u\Delta_gu+S_gu^2{\rm dvol}_{g}}}{\left(\int_M{u^6{\rm dvol}_{g}}\right)^{1/3}} &\text{\rm by formulas conformal change}\\
&=\frac{\int_M{8|\nabla_gu|^2+S_gu^2{\rm dvol}_{g}}}{\left(\int_M{u^6{\rm dvol}_{g}}\right)^{1/3}} &\text{\rm integration by parts}\\
&\ge\frac{\int_M{S_gu^2{\rm dvol}_{g}}}{\left(\int_M{u^6{\rm dvol}_{g}}\right)^{1/3}} &\text{\rm $|\nabla_gu|^2\ge 0$}.
\end{align*}
By Hölder's inequality
\[
\left|\int_M{S_gu^2{\rm dvol}_{g}}\right|\le \left(\int_M{|S_g|^{3/2}{\rm dvol}_{g}}\right)^{2/3}\left(\int_M{(u^2)^3{\rm dvol}_{g}}\right)^{1/3}.
\]
Thus we can continue the chain of inequalities with
\[
\ge\frac{-\left(\int_M{|S_g|^{3/2}{\rm dvol}_{g}}\right)^{2/3}\left(\int_M{(u^2)^3{\rm dvol}_{g}}\right)^{1/3}}{\left(\int_M{u^6{\rm dvol}_{g}}\right)^{1/3}}=-\left(\int_M{|S_g|^{3/2}{\rm dvol}_g}\right)^{2/3}.
\]
This concludes the proof of the proposition.
\end{proof}

In general, the sigma invariant is difficult to compute and we only have a few examples of manifolds for which we know the exact value. Luckily, hyperbolic 3-manifolds belong to this class as proved by Anderson \cite[Section 2]{And} and Kleiner and Lott \cite{KL} using Perelman's work \cite{Pe1,Pe2}.

\begin{thm}[{see \cite[Proposition 93.10]{KL}}]
\label{thm:yamabe}
Let $(M,g_M)$ be a closed hyperbolic 3-manifold. Then
\[
\sigma(M)=\mc{E}(M,g_M)=-6\cdot{\rm vol}(M,g_M)^{2/3}.
\]
\end{thm}

\subsection{Interpolation between tubes and cusps}

Our strategy to prove Theorems \ref{thm:main} and \ref{thm:main'} involves a smooth interpolation between the metric on a hyperbolic tube and the metric of a hyperbolic cusp. 

In order to do so, one is immediately led to study metrics of the form
\[
g:={\rm d}r^2+a(r)^2{\rm d}\theta^2+b(r)^2{\rm d}y^2.
\]
on $[R_0,R_1]\times\mb{S}^1\times[0,1]$ where $a,b:[R_0,R_1]\to(0,\infty)$ are smooth functions. In fact, when we choose  $R_0=0,R_1=R$ and $a(r)=2\pi\sinh(r),b(r)=\ell\cosh(r)$ we obtain the metric of a hyperbolic tube of radius $R$ and core geodesic of length $\ell$. When we choose $R_0=-\infty,R_1=0$ and $a(r)=\ell_1e^r,b(r)=\ell_2e^r$ we obtain the metric of a hyperbolic cusp based on a flat cylinder with core curve of length $\ell_1$ and height $\ell_2$.

In this section, we summarize the curvature properties of these well-studied metrics. Denote by $\partial_r,\partial_\theta,\partial_y$ the standard coordinate frame. 

\begin{lem}[{see \cite[Lemma 2.3]{A} or \cite[Lemma 10]{BH}}]
\label{lem:sectional}
Let $g:={\rm d}r^2+a(r)^2{\rm d}\theta^2+b(r)^2{\rm d}y^2$ be a metric on $[R_0,R_1]\times\mb{S}^1\times[0,1]$. We have the following symmetric table
\[
\begin{array}{c|c c c c}
    K(\partial_i,\partial_j) & &\partial_r &\partial_\theta &\partial_y\\
    \hline
    & & & &\\
    \partial_r & &0 &-\frac{a''}{a} &-\frac{b''}{b}\\
    & & & &\\
    \partial_\theta & & &0 &-\frac{a'b'}{ab}\\
    & & & &\\
    \partial_y & & & &0\\
\end{array}
\]
As $\partial_r,\partial_\theta,\partial_y$ is an orthogonal frame, the scalar curvature of the metric $g$ is given by
\[
S_g=2(K(\partial_r,\partial_\theta)+K(\partial_r,\partial_y)+K(\partial_\theta,\partial_y))=-2\left(\frac{a''}{a}+\frac{b''}{b}+\frac{a'b'}{ab}\right).
\]
\end{lem}

In the next two sections, we produce suitable smooth transitions from tubes to cusps by carefully choosing $a,b$. We are going to use them for modifying the hyperbolic metrics of our initial manifold $M$ in a tubular neighborhood of the simple closed geodesic in the drillings case and in the cuspidal end of the manifold in the fillings case.  

\section{From tubes to cusps}

In this section we construct the functions $a_{\rm drill},b_{\rm drill}$ used in the drilling case.

\begin{lem}
\label{lem:functiondrill}
There exists a constant $c>0$ such that the following holds. For every $R>0$ there exist functions $a_{\rm drill},b_{\rm drill}:\mb{R}\to\mb{R}$ such that 
\begin{enumerate}
    \item{$a_{\rm drill},b_{\rm drill}>0$.}
    \item{$a_{\rm drill}(r)=\sinh(r),b_{\rm drill}(r)=\cosh(r)$ on $(2R/3,\infty)$.}
    \item{$a_{\rm drill}(r)=\sinh(R)e^r,\cosh(R)b_{\rm drill}(r)=e^r$ on $(-\infty,R/3)$.}
    \item{On $[R/3,2R/3]$ we have
    \[
    \left|\frac{a_{\rm drill}''}{a_{\rm drill}}+\frac{b_{\rm drill}''}{b_{\rm drill}}+\frac{a_{\rm drill}'b_{\rm drill}'}{a_{\rm drill}b_{\rm drill}}\right|\le c\frac{1}{R^2}.
    \]
    }
    \item{We have $a_{\rm drill}b_{\rm drill}\le \sinh(R)\cosh(R)e^{2r}$.}
\end{enumerate}
\end{lem}

\begin{proof}
Let $\eta:\mb{R}\to\mb{R}$ be a non-decreasing smooth bump function such that 
\begin{itemize}
    \item{$\eta=0$ on $(-\infty,1/3]$ and $\chi=1$ on $[2/3,\infty)$.}
    \item{$|\eta'|\le c_1$ and $|\eta''|\le c_2$.}
\end{itemize} 
Set $\chi(r):=\eta(r/R)$. Clearly 
\begin{itemize}
    \item{$\chi=0$ on $(-\infty,R/3]$ and $\chi=1$ on $[2R/3,\infty)$.}
    \item{$|\chi'|\le c_1/R$ and $|\chi''|\le c_2/R^2$.}
\end{itemize} 
 
Define 
\begin{align*}
\phi(r):=(1-\chi(r))\log(\sinh(R)e^r)+\chi(r)\log(\sinh(r))\\
\psi(r):=(1-\chi(r))\log(\cosh(R)e^r)+\chi(r)\log(\cosh(r)).
\end{align*}

Set $a_{\rm drill}(r)=e^{\phi(r)}$ and $b_{\rm drill}(r)=e^{\psi(r)}$. 

This automatically ensures that properties (1), (2), and (3) are satisfied.

Let us now discuss property (4). We have
\[
\frac{a_{\rm drill}'}{a_{\rm drill}}=\phi',\frac{a_{\rm drill}''}{a_{\rm drill}}=(\phi')^2+\phi''\quad\text{\rm and }\quad \frac{b'}{b_{\rm drill}}=\psi',\frac{b_{\rm drill}''}{b_{\rm drill}}=(\psi')^2+\psi''.
\] 

We now compute the first derivatives
\begin{align*}
\phi'=\chi'(r)\log\left(\frac{\sinh(r)}{\sinh(R)e^r}\right)+\chi(r)(\coth(r)-1)+1,\\
\psi'=\chi'(r)\log\left(\frac{\cosh(r)}{\cosh(R)e^r}\right)+\chi(r)(\tanh(r)-1)+1.
\end{align*}

The terms with $\chi'$ coefficient have the following behavior. As $\sinh(x)/e^x$ is increasing, $\cosh(x)/e^x$ is decreasing, and $r\in[R/3,2R/3]$, we have
\[
\frac{\sinh(r)}{\sinh(R)e^r}\in\left(\frac{\sinh(R/3)}{\sinh(R)e^{R/3}},1\right)\quad\text{\rm and}\quad\frac{\cosh(r)}{\cosh(R)e^r}\in\left(\frac{\cosh(2R/3)}{\cosh(R)e^{2R/3}},1\right).
\]

On $[0,1]$, the functions $\frac{\sinh(y/3)}{\sinh(y)e^{y/3}},\frac{\cosh(2y/3)}{\cosh(y)e^{2y/3}}$ are decreasing, so they are bounded from below by 
\[
c_3=\frac{\sinh(1/3)}{\sinh(1)e^{1/3}},c_4=\frac{\cosh(2/3)}{\cosh(1)e^{2/3}}\in (0,1).
\]
respectively. Thus,
\[
\left|\log\left(\frac{\sinh(r)}{\sinh(R)e^r}\right)\right|\le|\log(c_3)|\quad\text{\rm and}\quad\left|\log\left(\frac{\cosh(r)}{\cosh(R)e^r}\right)\right|\le|\log(c_4)|.
\]

Consider then the terms with $\chi$ coefficient. On $[0,1]$ we have
\[
\coth(r)-1=\frac{2}{1-e^{-2r}}\in\left(0,\frac{1}{r}\right)\quad\text{\rm and}\quad\tanh(r)-1=-\frac{2}{1+e^{-2r}}\in\left(-2,-1\right).
\]
Therefore, for $r\in[R/3,2R/3]$, we get
\[
\left|\coth(r)-1\right|\le\frac{3}{R}\quad\text{\rm and}\quad\left|\tanh(r)-1\right|\le 2.
\]

Using the fact $|\chi'|\le c_1/R$ we get
\begin{align*}
|\phi'(r)| &\le\frac{c_1}{R}|\log(c_3)|+\frac{2}{R}+1=\frac{1}{R}\left(c_1|\log(c_3)|+2+R\right)\le\frac{c_5}{R},\\
|\psi'(r)| &\le\frac{c_1}{R}|\log(c_4)|+2+1=\frac{1}{R}\left(c_1|\log(c_4)|+3R\right)\le\frac{c_6}{R}
\end{align*}
where $c_5=c_1|\log(c_3)|+3$ and $c_6=c_1|\log(c_4)|+3$. 

As for the second derivatives, we have
\begin{align*}
\phi''=\chi''(r)\log\left(\frac{\sinh(r)}{\sinh(R)e^r}\right)+2\chi'(r)(\coth(r)-1)-\chi(r)\frac{1}{\sinh(r)^2}\\
\psi''=\chi''(r)\log\left(\frac{\cosh(r)}{\cosh(R)e^r}\right)+2\chi'(r)(\tanh(r)-1)+\chi(r)\frac{1}{\cosh(r)^2}.
\end{align*}

Using the fact $|\chi'|\le c_1/R,|\chi''|\le c_2/R^2$ and that
\[
\frac{1}{\sinh(r)^2}\le\frac{1}{r^2}\quad\text{\rm and}\quad\frac{1}{\cosh(r)^2}\le 1,
\]
we get
\begin{align*}
|\phi''(r)| &\le\frac{c_2}{R^2}|\log(c_3)|+2\frac{c_1}{R^2}+\frac{9}{R^2}=\frac{c_7}{R^2},\\
|\psi''(r)| &\le\frac{c_2}{R^2}|\log(c_4)|+4\frac{c_1}{R}+1=\frac{1}{R^2}\left(c_2|\log(c_4)|+4c_1R+R^2\right)\le\frac{c_8}{R^2}
\end{align*}
where $c_7=c_2|\log(c_3)|+2c_1+9$ and $c_8=c_2|\log(c_4)|+4c_1+1$.

In conclusion
\[
\frac{a_{\rm drill}''}{a_{\rm drill}}+\frac{b_{\rm drill}''}{b_{\rm drill}}+\frac{a_{\rm drill}'b_{\rm drill}'}{a_{\rm drill}b_{\rm drill}}=(\phi')^2+(\psi')^2+\phi'\psi'+\phi''+\psi''\le\frac{c}{R^2}
\]
where $c=c_5^2+c_6^2+c_5c_6+c_7+c_8$.

We need to check (5).
\begin{align*}
\phi(r)+\psi(r) &=(1-\chi(r))\log(e^{2r})+\chi(r)\log\left(\sinh(r)\cosh(r)\right) &\\
 &=\log(\sinh(R)\cosh(R)e^{2r})+\chi(r)\log\left(\frac{\sinh(r)\cosh(r)}{e^{2r}}\right) &\\
 &\le\log(\sinh(R)\cosh(R)e^{2r})
\end{align*}
as $\sinh(r)\cosh(r)/e^{2r}<1$. This concludes the proof of the lemma.
\end{proof}


\section{From cusps to tubes}

In this section we construct the functions $a_{\rm fill},b_{\rm fill}$ used in the filling case.

\begin{lem}
\label{lem:functionfill}
There exists a constant $c>0$ such that the following holds. For every $\ell_1>2\pi,\ell_2>0$ there exist functions $a_{\rm fill},b_{\rm fill}:\mb{R}\to\mb{R}$ that satisfy the following properties. Let $R=\log(\ell_1/\pi)$ and $\kappa=2\pi\ell_2/\ell_1$.
\begin{enumerate}
    \item{$a_{\rm fill},b_{\rm fill}>0$.}
    \item{$a_{\rm fill}(r)=2\pi\sinh(r+R),b(r+R)=\kappa\cosh(r+R)$ on $(-R,-2\delta)$.}
    \item{$a_{\rm fill}(r)=\ell_1e^r,b(r)=\ell_2 e^r$ on $(-\delta,0)$.}
    \item{We have
    \[
    \left|\frac{a_{\rm fill}''}{a_{\rm fill}}+\frac{b_{\rm fill}''}{b_{\rm fill}}+\frac{a_{\rm fill}'b_{\rm fill}'}{a_{\rm fill}b_{\rm fill}}-3\right|\le c\frac{\pi^4}{\ell_1^4}.
    \]
    }
    \item{We have 
    \[
    a_{\rm fill}b_{\rm fill}\le\ell_1\ell_2e^{2r}.
    \]
    }
\end{enumerate}
\end{lem}

\begin{proof}
Let $\chi:\mb{R}\to\mb{R}$ be a smooth non-decreasing function such that $\chi(r)=0$ on $(-\infty,-\log(2)/2)$ and $\chi(r)=1$ on $(-\log(2)/4,\infty)$ whose first and second derivatives are bounded by $|\chi'|\le c_1$ and $|\chi''|\le c_2$. Define 
\begin{align*}
\phi(r):=(1-\chi(r))\log(2\pi\sinh(r+R))+\chi(r)\log(\ell_1e^r)\\
\psi(r):=(1-\chi(r))\log(\kappa\cosh(r+R))+\chi(r)\log(\ell_2e^r).
\end{align*}

Set $a_{\rm fill}(r)=e^{\phi(r)}$ and $b_{\rm fill}(r)=e^{\psi(r)}$.

This automatically ensures that properties (1), (2), and (3) are satisfied. 

Let us now discuss property (4). We have
\[
\frac{a_{\rm fill}'}{a_{\rm fill}}=\phi',\frac{a_{\rm fill}''}{a_{\rm fill}}=(\phi')^2+\phi''\quad\text{\rm and }\quad \frac{b_{\rm fill}'}{b_{\rm fill}}=\psi',\frac{b_{\rm fill}''}{b_{\rm fill}}=(\psi')^2+\psi''.
\] 
So we need to control three terms. The first is $(\phi')^2+(\psi')^2$ which we conveniently write as
\begin{align*}
(\phi')^2+(\psi')^2 &=(1+(\phi'-1))^2+(1+(\psi'-1)^2\\
 &=2+2[(\phi'-1)+(\psi'-1)]+(\phi'-1)^2+(\psi'-1)^2.
\end{align*}
The second is $\phi'\psi'$, which we write as
\begin{align*}
\phi'\psi' &=(1+(\phi'-1))(1+(\psi'-1))\\
 &=1+[(\phi'-1)+(\psi'-1)]+(\phi'-1)(\psi'-1).
\end{align*}
The third is $\phi''+\psi''$. Let us begin by computing the first derivative. 
\begin{align*}
\phi'(r)-1 &=-\chi'(r)\log\left(\frac{2\pi}{\ell_1}\frac{\sinh(r+R)}{e^r}\right)+(1-\chi(r))(\coth(r+R)-1)\\
\psi'(r)-1 &=-\chi'(r)\log\left(\frac{\kappa}{\ell_2}\frac{\cosh(r+R)}{e^r}\right)+(1-\chi(r))(\tanh(r+R)-1).
\end{align*}

The terms with $\chi'(r)$ coefficient can be simplified using the definition of $R=\log(\ell_1/\pi)$ and $\kappa=2\pi\ell_2/\ell_1$ as follows. 
\[
\frac{2\pi}{\ell_1}\frac{\sinh(r+R)}{e^r}=\frac{\pi e^R}{\ell_1}\left(1-e^{-2(r+R)}\right)=1-e^{-2(r+R)}
\]
and
\[
\frac{\kappa}{\ell_2}\frac{\cosh(r+R)}{e^r}=\frac{\kappa e^R}{2\ell_2}\left(1+e^{-2(r+R)}\right)=1+e^{-2(r+R)}.
\]

The terms with the $1-\chi(r)$ coefficient can be better understood by writing 
\[
\coth(r+R)-1=\frac{2e^{-2(r+R)}}{1-e^{-2(r+R)}}
\]
and 
\[
\tanh(r+R)-1=-\frac{2e^{-2(r+R)}}{1+e^{-2(r+R)}}.
\]

Putting together the above, we get
\begin{align*}
\phi'(r)-1 &=-\chi'(r)\log\left(1-e^{-2(r+R)}\right)+(1-\chi(r))\frac{2e^{-2(r+R)}}{1-e^{-2(r+R)}},\\
\psi'(r)-1 &=-\chi'(r)\log\left(1+e^{-2(r+R)}\right)-(1-\chi(r))\frac{2e^{-2(r+R)}}{1+e^{-2(r+R)}}.
\end{align*}

Recall that $r\in[-\log(2)/2,-\log(2)/4]$ and observe that 
\[
e^{-2(r+R)}=\frac{\pi^2}{\ell^2}e^{-2r}<\frac{1}{4}e^{-2r}<\frac{1}{2}.
\]
As $|\log(1\pm x)|\le x$ for $x\in(0,1)$ and $1/(1\pm x)\le 3/2\mp1/2$ for $x\in(0,1/2)$, we have $|\phi'(r)-1|\le c_3e^{-2R},|\psi'(r)-1|\le c_4e^{-2R}$ where $c_3=2(c_1+4),c_4=2(c_1+2)$ only depends on $\chi$. 

By the above formulas, we get
\[
(\phi'(r)-1)+(\psi'(r)-1)=-\chi'(r)\log\left(1-e^{-4(r+R)}\right)+(1-\chi(r))\frac{4e^{-4(r+R)}}{1-e^{-4(r+R)}}.
\]
Using the same controls on $\log(1\pm x)$ and $1/(1\pm x)$ as above (note that $e^{-4(r+R)}<e^{-2(r+R)}$) we deduce that $|(\phi'(r)-1)+(\psi'(r)-1)|\le c_5e^{-4R}$ where $c_5=4(c_1+8)$ only depends on $\chi$.

As for the second derivatives, we have
\begin{align*}
\phi''(r) &=-\chi''(r)\log\left(\frac{2\pi}{\ell_1}\frac{\sinh(r+R)}{e^r}\right)-2\chi'(r)(\coth(r+R)-1)-(1-\chi(r))\frac{1}{\sinh(r+R)^2},\\
\psi''(r) &=-\chi''(r)\log\left(\frac{2\pi}{\ell_1}\frac{\cosh(r+R)}{e^r}\right)-2\chi'(r)(\tanh(r+R)-1)+(1-\chi(r))\frac{1}{\cosh(r+R)^2}.
\end{align*}
Simplifying the expressions as above, we obtain
\begin{align*}
\phi''(r) &=-\chi''(r)\log\left(1-e^{-2(r+R)}\right)-2\chi'(r)\frac{2e^{-2(r+R)}}{1-e^{-2(r+R)}}-(1-\chi(r))\left(\frac{2e^{-(r+R)}}{1-e^{-2(r+R)}}\right)^2,\\
\psi''(r) &=-\chi''(r)\log\left(1+e^{-2(r+R)}\right)+2\chi'(r)\frac{2e^{-2(r+R)}}{1+e^{-2(r+R)}}+(1-\chi(r))\left(\frac{2e^{-(r+R)}}{1+e^{-2(r+R)}}\right)^2.
\end{align*}
Summing the two terms, we get
\[
\phi''(r)+\psi''(r)=-\chi''(r)\log\left(1-e^{-4(r+R)}\right)-2\chi'(r)\frac{4e^{-4(r+R)}}{1-e^{-4(r+R)}}-(1-\chi(r))\frac{16e^{-4(r+R)}}{(1-e^{-4(r+R)})^2}.
\]
Again, using $|\log(1\pm x)|\le x$ and $1/(1\pm x)\le 2$ we deduce that there is a constant $c_6=4(c_2+16c_1+32)$ such that $|\phi''(r)+\psi''(r)|\le c_6e^{-4R}$. Putting all the terms together, we conclude the following inequality
\begin{align*}
&\left|(\phi')^2+(\psi')^2+\phi'\psi'+\phi''+\psi''-3\right|\\
&=\left|2[(\phi'-1)+(\psi'-1)]+(\phi'-1)^2+(\psi'-1)^2+[(\phi'-1)+(\psi'-1)]+(\phi'-1)(\psi'-1)+\phi''+\psi''\right|\\
&\le 3|(\phi'-1)+(\psi'-1)|+|\phi'-1||\psi'-1|+|\phi'-1|^2+|\psi'-1|^2+|\phi''+\psi''|\\
&\le (3c_5+c_3c_4+c_3^2+c_4^2+c_6)e^{-4R}=c\pi^4/\ell^4
\end{align*}
where $c=3c_5+c_3c_4+c_3^2+c_4^2+c_6$ only depends on $\chi$.

Lastly, we need to check (5). We have
\begin{align*}
\phi(r)+\psi(r) &=(1-\chi(r))\log\left(2\pi\kappa\sinh(r+R)\cosh(r+R)\right)+\chi(r)\log(\ell_1\ell_2e^{2r}) &\\
 &=(1-\chi(r))\log\left(\frac{2\pi\kappa e^{2R}}{4}e^{2r}(1-e^{-4(r+R)})\right)+\chi(r)\log(\ell_1\ell_2e^{2r}). &
\end{align*}
Using $\frac{2\pi\kappa e^{2R}}{4}=\ell_1\ell_2$ we continue the chain of equalities as
\begin{align*} 
 &=(1-\chi(r))\log\left(\ell_1\ell_2e^{2r}(1-e^{-4(r+R)})\right)+\chi(r)\log(\ell_1\ell_2e^{2r}) &\\
 &=\log(\ell_1\ell_2e^{2r})+(1-\chi(r))\log\left(1-e^{-4(r+R)}\right)\le \log(\ell_1\ell_2e^{2r}). &
\end{align*}
This concludes the proof of the lemma.
\end{proof}


\section{Metrics on tubes}

Having defined the functions $a_{\rm fill},b_{\rm fill}$, we use them to prove the following extension result, which might be of independent interest (it is a quantitative version of the $2\pi$-Theorem, compare with \cite[Theorem 9]{BH}, \cite[Proposition 3.1]{CL}, and \cite[Theorem 2.1]{FKP}).

\begin{pro}
\label{pro:metric on tube}
There exists a constant $c>0$ such that the following holds. Let $V$ be a solid torus. Suppose that the boundary $\partial V$ is equipped with a flat metric of area $A$. Assume that the length of the flat geodesic representative on $\partial V$ of the meridian $\mu\subset\partial V$ is $\ell>2\pi$. Then there exists a smooth Riemannian metric $(V,g)$ with the following properties.
\begin{itemize}
    \item{The metric $g$ is a hyperbolic cusp metric in a collar of $\partial V$, the boundary $\partial V$ is a horospherical section of the cusp, and the restriction of $g$ to the boundary agrees with the prescribed flat metric on $\partial V$.}
    \item{Denote by $S_g$ the scalar curvature of the metric $g$. We have
    \[
    \int_V{\left(\frac{|S_g|}{6}\right)^{3/2}{\rm dvol}_g}\le\frac{A}{2}\left(1-\frac{\pi^2}{\ell^2}\right)\left(1+c\frac{\pi^4}{\ell^4}\right).
    \]
    }
\end{itemize}    
\end{pro}

\begin{proof}
Let $R:=\log(\ell/\pi)$.

Consider the normal covering $W\to\partial V$ corresponding to ${\rm ker}(\pi_1(\partial V)\to\pi_1(V))=\mb{Z}\mu$ (the loop $\mu\subset\partial V$ is a flat representative of the meridian). By covering theory, the meridian $\mu$ lifts to a simple closed curve on $W$. We denote by $\delta:W\to W$ a generator of the deck group $\pi_1(\partial V)/\mb{Z}\mu$, so that $\partial V=W/\langle\delta\rangle$. 

Equipped with the induced flat metric, $W$ is a flat cylinder which we can isometrically identify with 
\[
(\mb{S}^1\times\mb{R},\ell_1^2{\rm d}\theta^2+{\rm d}y^2)
\]
where $\mb{S}^1=\mb{R}/\mb{Z}$. Since $\mu$ lifts homeomorphically to a simple closed geodesic on $W$ of length $\ell$, we must have $\ell_1=\ell$. Under the identification $W\simeq\mb{S}^1\times\mb{R}$ the deck transformation $\delta$ translates into the isometry 
\[
\tau:\left(\theta,y\right)\to\left(\theta+\alpha,y+s\right)
\]
and we have $\partial V=W/\langle\delta\rangle=\mb{S}^1\times\mb{R}/\langle\tau\rangle$. In order to compute $s$ we observe the following. Every fundamental domain of $\tau$ has area equal to $A$, the area of $\partial V$. By the explicit formula for $\tau$, a fundamental domain is given by $\mb{S}^1\times[0,s]$ which has area $\ell_1s$. Thus $s=A/\ell_1=A/\ell$.

By the above discussion, the isometric diffeomorphism 
\[
f:\mb{S}^1\times\mb{R}/\langle\tau\rangle\to\partial V
\]
maps $\mb{S}^1\times\{0\}$ a parallel copy of $\mu$, the geodesic representative on $\partial V$ of the meridian. We now extend this diffeomorphism to a diffeomorphism of a suitable solid torus with boundary $\mb{S}^1\times\mb{R}/\langle\tau\rangle$ and $V$. In order to do so, observe that $\tau$ extends to a diffeomorphism of 
\[
U:=\mb{S}^1\times\mb{R}\times[-R,0]\,\left/\,(\theta,y,-R)\sim(\theta',y,-R)\right.
\]
which we keep denoting by $\tau$ with a slight abuse of notation, by the formula
\[
\tau(\theta,y,r)=\left(\theta+\alpha,y+s,r\right).
\]
The group $\langle\tau\rangle$ generated by this extension acts freely and properly discontinuously on $U$. The quotient $U/\langle\tau\rangle$ is a solid torus with boundary $\mb{S}^1\times\mb{R}\times\{0\}/\langle\tau\rangle$ and meridian given by $\mb{S}^1\times\{0\}\times\{0\}$. As $f:\partial U/\langle\tau\rangle\to\partial V$ maps the meridian of the source to the meridian of the target, it can be extended to a diffeomorphism $U/\langle\tau\rangle\to V$. 

From now on we work on $U$ whose boundary $\mb{S}^1\times\mb{R}\times\{0\}$ is equipped with the metric $\ell_1^2{\rm d}\theta^2+{\rm d}y^2$. We will consider metrics of the form $g={\rm d}r^2+a(r)^2{\rm d}\theta^2+b(r)^2{\rm d}y^2$. All these metrics are $\tau$-invariant. Hence, they descend to Riemannian metrics on $U/\langle\tau\rangle$. Note that $g$ is potentially singular at the $r=-R$ locus (which is a line). It is well-known that $g$ is smooth everywhere if $a'(-R)=2\pi$ (see \cite{BH} or \cite{FKP}). We choose $a(r)=a_{\rm fill}(r)$ and $b(r)=b_{\rm fill}(r)$ as given by Lemma \ref{lem:functionfill} with input $\ell_1=\ell$ and $\ell_2=1$. Note that property (2) of Lemma \ref{lem:functionfill} guarantees that $a'_{\rm fill}(-R)=2\pi$. Note also that property (4) of Lemma \ref{lem:functionfill} together with Lemma \ref{lem:sectional} implies that the scalar curvature $S_g$ of the resulting metric satisfies
\[
\left|\left(\frac{S_g}{6}\right)^{3/2}-1\right|\le\left|\left(1+c\frac{\pi^4}{\ell^4}\right)^{3/2}-1\right|\le \frac{3}{2}c\frac{\pi^4}{\ell^4}.
\]

We now compute 
\[
\int_V{\left|\frac{S_g}{6}\right|^{3/2}{\rm dvol}_g}.
\]
As $S_g$ is close to 1, we first compare it with ${\rm vol}(V,g)$ controlling the difference. Then we show that ${\rm vol}(V,g)$ is bounded by $(1-\pi^2/\ell^2)A/2$. In order to do these computations, we work on the fundamental domain $F:=\mb{S}^1\times[0,s]\times[-R,0]$ for the action of $\tau$ on $\mb{S}^1\times\mb{R}\times[-R,0]$. We have
\[
\left|\int_{V}{\left(\frac{|S_g|}{6}\right)^{3/2}\,{\rm dvol}_g}-{\rm vol}(V,g)\right|\le\int_V{\left|\left(\frac{|S_g|}{6}\right)^{3/2}-1\right|\,{\rm dvol}_g}\le \frac{3}{2}c\frac{\pi^4}{\ell^4}{\rm vol}(V,g).
\]
Let us compute the volume. By Lemma \ref{lem:functionfill} (5), we have $a_{\rm fill}b_{\rm fill}\le\ell e^{2r}$. Thus
\begin{align*}
{\rm vol}(V,g) &={\rm vol}(F,g)\\
&=\int_0^{A/\ell}\int_0^1\int_{-R}^0{a_{\rm fill}b_{\rm fill}\,{\rm d}r\,{\rm d}\theta\,{\rm d}y} &\\
 &=\frac{A}{\ell}\int_{-R}^0{a_{\rm fill}b_{\rm fill}{\rm d}r}\\
 &\le\frac{A}{\ell}\int_{-R}^0{\ell e^{2r}{\rm d}r}=\frac{A}{2}\left(1-\frac{\pi^2}{\ell^2}\right).
\end{align*}
In conclusion
\begin{align*}
\int_{V}{\left(\frac{|S_g|}{6}\right)^{3/2}\,{\rm dvol}_g} &=\left(\int_{V}{\left(\frac{|S_g|}{6}\right)^{3/2}\,{\rm dvol}_g}-{\rm vol}(V,g)\right)+{\rm vol}(V,g)\\
&\le\left(1+\frac{3}{2}c\frac{\pi^4}{\ell^2}\right){\rm vol}(V,g)\le\left(1+\frac{3}{2}c\frac{\pi^4}{\ell^2}\right)\frac{A}{2}\left(1-\frac{\pi^2}{\ell^2}\right)
\end{align*}
as desired.
\end{proof}

\section{Proof of Theorem \ref{thm:main}}

There is a tubular neighborhood $T$ of $\gamma$, parametrized by the so-called Fermi coordinates, where the hyperbolic metric of $M$ can be written as
\[
g_M={\rm d}r^2+4\pi^2\sinh(r)^2{\rm d}\theta^2+\ell^2\cosh(r)^2{\rm d}y^2
\]
where $y\in[0,1]$ is a length parameter along $\gamma$ while $r\in[0,R],\theta\in\mb{S}^1=\mb{R}/\mb{Z}$ are respectively the radial and angular parameters around $\gamma$. 

Heuristically speaking, the basic picture is as follows. Ideally, one should have $\sigma(M-\gamma)=-6^{3/2}{\rm vol}(M-\gamma)^{2/3}$ (similar to Theorem \ref{thm:yamabe}) and would like to exploit the inequality $\sigma(M-\gamma)\ge-(\int_{M-\gamma}{|S_g|^{3/2}{\rm dvol}_g})^{2/3}$ (similar to Proposition \ref{pro:integral scalar}) to compare ${\rm vol}(M-\gamma)$ to the integral of the scalar curvature for a Riemannian metric we obtain by modifying of the hyperbolic metric $g_M$ on $T-\gamma$ (replacing with one of those defined using Lemma \ref{lem:functiondrill}). This strategy cannot be carried out exactly as stated since Theorem \ref{thm:yamabe} and Proposition \ref{pro:integral scalar} only work for closed manifolds ($M-\gamma$ has a cusp). However, this is just a technical obstacle that we will bypass using a Dehn filling trick (as in \cite{AST}).

Let us start with the proof. It is convenient to define the following. We will call {\em $[r_0,r_1]$-shell} the submanifold of $T$ that in Fermi coordinates corresponds to the region where the radial coordinate $r$ varies in the interval $[r_0,r_1]$.

Fix $\ep>0$ arbitrary. According to our strategy, we replace the metric $g_M$ on the $[\ep,R]$-shell $S_\ep$ with the metric
\[
g_{S_\ep}:={\rm d}r^2+4\pi^2a_{\rm drill}(r)^2{\rm d}\theta^2+\ell^2b_{\rm drill}(r)^2{\rm d}y^2
\]
where $a_{\rm drill},b_{\rm drill}$ are the functions provided by Lemma \ref{lem:functiondrill} with input $R$. Note that $g_{S_\ep}$ coincides with $g_M$ on the $[2R/3,R]$-shell and with
\[
g_C:={\rm d}r^2+4\pi^2\sinh(R)^2e^{2r}{\rm d}\theta^2+\ell^2\cosh(R)^2e^{2r}{\rm d}y^2
\]
in the $[\ep,R/3]$-shell. The metric $g_C$ naturally extends to every $r\le\ep$ with the same formula. The extension is isometric to a hyperbolic cusp where the slices $r={\rm constant}$ are horospherical tori.

The next step is just a trick to remedy the fact that the sigma invariant computation only works for closed hyperbolic manifolds. What we do is to realize $M-\gamma$ as a geometric limit of closed hyperbolic manifolds using long Dehn fillings. By standard properties of the geometric convergence, the volume of $M-\gamma$ can be computed as the limit of the volume of the approximating hyperbolic manifolds.

Consider the smaller torus $T_\ep:=T-S_\ep$. Choose an infinite sequence of pairwise distinct essential simple closed curves $\mu_n$ on the torus $\partial T_\ep$ and define $M_n$ to be the Dehn filling of $M-T_\ep$ with filling slope $\mu_n$, that is 
\[
M_n=(M-T_\ep)\cup_{\mu_n}\mb{D}^2\times\mb{S}^1
\]
where the gluing map $\mb{D}^2\times\mb{S}^1\to\partial T_\ep$ is completely determined (up to isotopy) by the requirement that a meridian of $\mb{D}^2\times\mb{S}^1$ is mapped to $\mu_n$. By Thurston's Hyperbolic Dehn Filling Theorem (see for example \cite[Chapter 15]{Martelli}), for large enough $n$ we have that $M_n$ admits a hyperbolic metric $g_{M_n}$ whose volume ${\rm vol}(M_n,g_{M_n})$ converges to the volume of $(M-\gamma,g_{M-\gamma})$ where $g_{M-\gamma}$ denotes the hyperbolic metric on $M-\gamma$.

By Proposition \ref{pro:metric on tube}, for every $n$ large enough we can smoothly extend the metric $g_C$ to a metric $g_{\mb{D}^2\times\mb{S}^1,n}$ on the added solid torus $\mb{D}^2\times\mb{S}^1$ so that
\[
\int_{\mb{D}^2\times\mb{S}^1}{|S_{g_n}|^{3/2}{\rm dvol}_{g_n}}\to 6^{3/2}\cdot\frac{{\rm Area}(\partial T_\ep,g_C)}{2}=e^\ep\sinh(R)\cosh(R)\pi\ell.
\]

By Theorem \ref{thm:yamabe}, we have $-6\cdot{\rm vol}(M,g_{M_n})^{2/3}=\sigma(M_n)$.

By Proposition \ref{pro:integral scalar} we have
\[
\sigma(M_n)\ge-\left(\int_{M_n}{|S_{g_n}|^{3/2}{\rm dvol}_{g_n}}\right)^{2/3}
\]
for any Riemannian metric $g_n$ on $M_n$. In particular, this holds for the Riemannian metric defined by 
\[
g_n:=\left\{
\begin{array}{ll}
 g_M &\text{\rm on $M-T$},\\
     g_{S_\ep} &\text{\rm on $S_\ep$},\\
     g_{\mb{D}^2\times\mb{S}^1,n} &\text{\rm on $\mb{D}^2\times\mb{S}^1$}.\\
\end{array}
\right.
\]
Thus,
\[
6^{3/2}\cdot{\rm vol}(M,g_{M_n})\le\int_{M_n}{|S_{g_n}|^{3/2}{\rm dvol}_{g_n}}.
\]

We split the integral as
\begin{align*}
\int_{M_n}{|S_{g_n}|^{3/2}{\rm dvol}_{g_n}} &=\int_{M-T}{|S_{g_M}|^{3/2}{\rm dvol}_{g_M}}+\int_{S_\ep}{|S_{g_S}|^{3/2}{\rm dvol}_{g_S}}\\
 &+\int_{\mb{D}^2\times\mb{S}^1}{|S_{g_{\mb{D}^2\times\mb{S}^1,n}}|^{3/2}{\rm dvol}_{g_{\mb{D}^2\times\mb{S}^1,n}}}.
\end{align*}

As the metric $g_n$ coincides with the hyperbolic metric on $M-T$, the first term is
\[
\int_{M-T}{|S_g|^{3/2}{\rm dvol}_g}=6^{3/2}\cdot{\rm vol}(M-T)\le 6^{3/2}{\rm vol}(M).
\]

On $S_\ep$, we compute 
\begin{align*}
&\int_0^\ell\int_0^1\int_\ep^R{|S_g|^{3/2}a_{\rm drill}b_{\rm drill}\,{\rm d}r\,{\rm d}\theta\,{\rm d}y}\\
&\le 2\pi\ell(2c)^{3/2}\frac{1}{R^3}\int_\ep^R{a_{\rm drill}b_{\rm drill}\,{\rm d}r} &\text{by Lemma \ref{lem:functiondrill} (4)}\\
&\le 2\pi\ell(2c)^{3/2}\frac{1}{R^3}\int_\ep^R{\sinh(R)\cosh(R)e^{2r}\,{\rm d}r} &\text{by Lemma \ref{lem:functiondrill} (5)}\\
&=(2c)^{3/2}\pi\ell\frac{\sinh(R)\cosh(R)}{R^3}(e^{2R}-e^{2\ep}).
\end{align*}

Lastly, on $\mb{D}^2\times\mb{S}^1$ we have
\[
\int_{\mb{D}^2\times\mb{S}^1}{|S_{g_{\mb{D}^2\times\mb{S}^1,n}}|^{3/2}{\rm dvol}_{g_{\mb{D}^2\times\mb{S}^1,n}}}\to 6^{3/2}e^\ep\sinh(R)\cosh(R)\pi\ell.
\]

Putting together the three terms and dividing by $6^{3/2}$, we obtain
\[
{\rm vol}(M_n,g_{M_n})\le{\rm vol}(M)+\frac{(2c)^{3/2}}{6^{3/2}}\pi\ell\frac{\sinh(R)\cosh(R)}{R^3}(e^{2R}-e^{2\ep})+e^\ep\sinh(R)\cosh(R)\pi\ell
\]
Passing to the limit as $n\to\infty$ and then as $\ep\to 0$, we get
\[
{\rm vol}(M-\gamma)\le{\rm vol}(M)+\pi\frac{\ell}{R}\left(\frac{(2c)^{3/2}}{6^{3/2}}\frac{\sinh(R)\cosh(R)(e^{2R}-1)}{R^2}+\frac{\sinh(R)\cosh(R)}{R}\right).
\]
In order to conclude, we just observe that that $\sinh(R)\cosh(R)(e^{2R}-1)/R^2\le c_1$ and that $\sinh(R)\cosh(R)/R\le c_2$ for $R\in (0,1)$ (where $c_1$ can be chosen to be roughly $c_1=11.59$ and $c_2=1.81$). This concludes the proof of the Theorem \ref{thm:main}.\qed


\section{The proof of Theorem \ref{thm:main'}}

Let $C$ be the cusp neighborhood bounded by the horospherical torus $\partial C$. By Proposition \ref{pro:metric on tube}, we can find a metric $g_\mu$ on the Dehn filling
\[
M_\mu=(M-C)\cup_\mu\mb{D}^2\times\mb{S}^1
\]
that agrees with the hyperbolic metric of $M$ on $M-C$ and satisfies
\[
\int_{\mb{D}^2\times\mb{S}^1}{|S_g|^{3/2}{\rm dvol}_g}\le 6^{3/2}\frac{A}{2}\left(1-\frac{\pi^2}{\ell^2}\right)\left(1+c\frac{\pi^4}{\ell^4}\right).
\]

By Theorem \ref{thm:yamabe}, we have $-6\cdot{\rm vol}(M_\mu,g_{M_\mu})^{2/3}=\sigma(M_\mu)$.

By Proposition \ref{pro:integral scalar} we have
\[
\sigma(M_\mu)\ge-\left(\int_{M_\mu}{|S_g|^{3/2}{\rm dvol}_g}\right)^{2/3}.
\]
for every Riemannian metric $g$ on $M_\mu$. In particular
\[
6^{3/2}\cdot{\rm vol}(M_\mu)\le \int_{M_\mu}{\left|S_{g}\right|^{3/2}{\rm dvol}_{g}}
\]

We use the metric $g=g_\mu$ and compute the integral
\[
\int_{M_\mu}{\left|S_{g_{\mu}}\right|^{3/2}{\rm dvol}_{g_{\mu}}}=\int_{M-C}{|S_{g_M}|^{3/2}{\rm dvol}_{g_M}}+\int_{\mb{D}^2\times\mb{S}^1}{\left|S_{g_{\mu}}\right|^{3/2}{\rm dvol}_{g_{\mu}}}.
\]

As the metric $g_{\mu}$ coincides with the hyperbolic metric on $M-C$, the first term is
\[
\int_{M-C}{|S_g|^{3/2}{\rm dvol}_g}=6^{3/2}\cdot{\rm vol}(M-C)=6^{3/2}\left({\rm vol}(M)-\frac{A}{2}\right).
\]

By Proposition \ref{pro:metric on tube}, on $\mb{D}^2\times\mb{S}^1$ we have 
\[
\int_{\mb{D}^2\times\mb{S}^1}{|S_g|^{3/2}{\rm dvol}_g}\le 6^{3/2}\frac{A}{2}\left(1-\frac{\pi^2}{\ell^2}\right)\left(1+c\frac{\pi^4}{\ell^4}\right).
\]

Putting together the terms and dividing by $6^{3/2}$, we obtain
\[
{\rm vol}(M_\mu)+\frac{A}{2}\left(\frac{\pi^2}{\ell^2}-c\frac{\pi^4}{\ell^4}+c\frac{\pi^6}{\ell^6}\right)\le{\rm vol}(M).
\]
This concludes the proof of the Theorem \ref{thm:main'}.\qed

\bibliographystyle{amsalpha.bst}
\bibliography{bibliography}

\Addresses

\end{document}